 \documentclass{amsart} %\documentclass[a4paper, 12pt]{article}
\usepackage{amssymb, latexsym, ,amsrefs, amsmath, amsthm, color, marginnote,hyperref,cancel}

 \usepackage{color}
\newtheorem{theorem}{Theorem}[section]
\newtheorem{lemma}[theorem]{Lemma}
\newtheorem{proposition}[theorem]{Proposition}
\newtheorem{corollary}[theorem]{Corollary}

\theoremstyle{definition}
\newtheorem{notation}[theorem]{Notation}
\newtheorem{definition}[theorem]{Definition}

\newtheorem{conjecture}[theorem]{Conjecture}

\theoremstyle{remark}
 \newtheorem{remark}[theorem]{Remark}

\numberwithin{equation}{section}

\newcommand{\Rtilde}{\widetilde{R}}
\newcommand{\Stilde}{\widetilde{S}}
\newcommand{\Ttilde}{\widetilde{T}}
\newcommand{\partialtilde}{\widetilde{\partial}_2}
\newcommand{\hhalf}{\textstyle{\frac{1}{2}}}
\newcommand{\half}{{\frac{1}{2}}}

\DeclareMathOperator{\gr}{{gr}}
\DeclareMathOperator{\GKdim}{{GKdim}}
\DeclareMathOperator{\Ext}{{Ext}}

\begin{document}
\title{The prime spectrum of the   Drinfeld  double of the Jordan plane}
\author{K. A. Brown}
\address{
School of Mathematics and Statistics, University of Glasgow, Glasgow G12 8QQ, Scotland}
\email{ken.brown@glasgow.ac.uk}
\thanks{}
\author{J. T. Stafford}
\address{School of Mathematics,  The University of Manchester,   Manchester M13 9PL,
	England}
\email{Toby.Stafford@manchester.ac.uk}
\thanks{Both authors are partially supported by Leverhulme Emeritus Fellowships, respectively EM-2017-081 and EM-2019-015. The first author thanks Nicolas Andruskiewitsch, Ivan Angiono and Hector Pena Pollastri for helpful comments and for sharing early versions of their work.}
\subjclass[2010]{Primary 16T05, 16D25; Secondary 16T20, 16S40,  17B37.}

%\date{\today}
 
 \begin{abstract}The Hopf algebra $\mathcal{D}$ which is the subject of this paper can be viewed as a Drinfeld double of the bosonisation of the Jordan plane. Its prime and primitive spectra are completely determined. As a corollary of this analysis it is shown that $\mathcal{D}$ satisfies the Dixmier-Moeglin Equivalence, leading to the formulation of a conjecture on the validity of this equivalence for pointed Noetherian Hopf algebras.\end{abstract}

\maketitle

\section{Introduction}\label{intro}
\noindent{\bf 1.1.} Throughout, $k$ will denote an   algebraically closed field of characteristic 0. The Hopf $k$-algebra $\mathcal{D}$ of the title was defined and some initial properties were derived in \cite{AP}, with further results in \cites{ADPP,AP2}. Our focus here is on the prime and primitive spectra of $\mathcal{D}$, which we completely determine. The Hopf algebra $\mathcal{D}$ is a pointed affine noetherian domain of Gelfand-Kirillov dimension 6 whose definition is recalled in $\S$\ref{notation}. It is a beautiful algebra with a number of striking properties which   make it worthy of study from several perspectives, three of which we briefly outline in $\S\S$1.3-1.5.  First we summarise our results and explain where they are located. 

\medskip

\noindent{\bf 1.2. Results.} It was proved in \cite[Theorem~4.10]{ADPP}, and  explained in detail here in Lemma~\ref{normal} 
and Theorem~\ref{before}(iv), that the centre $Z(\mathcal{D})$ of $\mathcal{D}$ is generated by elements $z, \omega$ and $\theta$ with $z\theta = \omega^2$. Thus $\mathrm{Maxspec}(Z(\mathcal{D}))$ has one singular point, namely 
$\mathfrak{m}_0 := \langle z,\omega,\theta \rangle$. There is one other distinctive maximal ideal of the centre, namely 
$$\mathfrak{m}_+ \; := \; Z(\mathcal{D}) \cap \mathcal{D}^+ \; = \; \langle z-16, \omega + 16, \theta - 16 \rangle, $$
where $\mathcal{D}^+$ is the augmentation ideal of $\mathcal{D}$. Finally, let $K$ denote the kernel of the Hopf algebra surjection $\pi:\mathcal{D} \longrightarrow U(\mathfrak{sl}(2,k))$, mentioned in $\S$1.3, so $K$ is a Hopf ideal which is described in Theorem~\ref{before}(i),(ii).

The main results of this paper 
are given by the following theorems and give a complete description of the prime and primitive ideals of $\mathcal{D}$. 
\begin{theorem}\label{main1} Retain the above notation.
  The primitive ideals of $\mathcal{D}$ are:
\begin{itemize}
\item[(I)] the maximal ideals $\mathfrak{m}\mathcal{D}$ for $\mathfrak{m} \in \mathrm{Maxspec}(Z(\mathcal{D}))$, with $\mathfrak{m} \neq \mathfrak{m}_0,\mathfrak{m}_+$; 
\item[(II)] the primitive ideals containing $\mathfrak{m}_+\mathcal{D}$, namely $\mathfrak{m}_+\mathcal{D}$ itself together with 
$$\mathcal{P} := \pi^{-1}(\mathrm{Privspec}(U(\mathfrak{sl}(2,k))));$$
\item[(III)] the unique prime ideal $P_0$ containing $\mathfrak{m}_0\mathcal{D}$, which has $ \mathfrak{m}_0\mathcal{D} = (P_0)^2 \subsetneq P_0$.
\end{itemize}
\end{theorem}

\begin{theorem}\label{main2}  In   the above notation,
  the non-primitive prime ideals of $\mathcal{D}$ are:
\begin{itemize}
\item[(A)] $\{0\},\,K$;
\item[(B)] the principal prime ideals $\mathfrak{p}\mathcal{D}$ for every height one prime $\mathfrak{p}$ of $Z(\mathcal{D})$ except $\mathfrak{p}_1 = \langle z,\omega\rangle$ and $\mathfrak{p}_2 = \langle \theta, \omega \rangle$;
\item[(C)] height one primes $P_1, P_2$, with $\mathfrak{p}_i\mathcal{D} = P_i^2 \subsetneq P_i$ for $i = 1,2$, with each $P_i$ generated by a normal (but not central) element. Moreover,  $P_1 + P_2 = P_0$.
\end{itemize}
\end{theorem}

\begin{theorem}\label{main3} Retain the above notation.
\begin{enumerate}
\item Every non-primitive prime is completely prime. Every primitive ideal is completely prime, except the co-Artinian maximal ideals (other than the counit), which form a subset of $\mathcal{P}$.
\item Every prime ideal $P$, apart from the co-Artinian maximal ideals and (possibly) $P_0$, has $\mathcal{D}/P$ birationally equivalent to a Weyl algebra $A_n(K)$, where  $1 \leq n \leq 2$ and $K$ is a field of transcendence degree at most 2 over $k$.
\item$\mathcal{D}$ satisfies the Dixmier-Moeglin Equivalence.
\end{enumerate}
\end{theorem}

Theorem~\ref{main1} is proved in Section~\ref{privspec},  see in particular Subsection~\ref{privshape}, while Theorem~\ref{main2} is proved in  Theorem~\ref{prime}.    Finally, Theorem~\ref{main3} is proved in Subsection~\ref{DM}.
Some questions and a conjecture (Conjecture \ref{wild}) are scattered through the paper.

\medskip

\noindent{\bf 1.2. Hopf algebras in duality.} The full Drinfeld double $\mathcal{D}(H) = H \bowtie H^{\circ}$  of an infinite dimensional 
Hopf algebra $H$ may often be unwieldy due to $H$ having ``too many'' finite dimensional representations and thus leading to an 
unmanageably large finite dual $H^{\circ}$. This has generated significant recent interest in constructing doubles 
$\mathcal{D}(H) := H \bowtie H'$ where $H'$ is some suitable Hopf subalgebra of $H^{\circ}$; see for example \cites{LL, BCJ}
 and the papers listed in $\S$1.1. Much is at present unclear: for example, what is an appropriate definition of a  ``suitable'' Hopf subalgebra $H'$; and
  does a suitable algebra $H'$ always exist?  The double $\mathcal{D}$ of the Jordan plane is a test case for these and other
   questions. In particular, some of the desirable properties exhibited by $\mathcal{D}$ may form a paradigm for what one might 
   aim for in defining doubles in more general settings.

\medskip

\noindent{\bf 1.3. Representation theory.} A striking feature of the representation theory of $\mathcal{D}$ is the
 fact, proved in \cite[Theorem~3.11]{ADPP} and recalled here in Theorem~\ref{before}(iii), that $U(\mathfrak{sl}(2,k))$ is a quotient
  Hopf algebra of $\mathcal{D}$ and the finite dimensional irreducible $\mathcal{D}$-modules are precisely the finite dimensional
   irreducible $U(\mathfrak{sl}(2,k))$-modules. This immediately suggests a plethora of questions, some of which are addressed 
   in \cite{AP2}, where Verma $\mathcal{D}$-modules are introduced. But many others remain untouched: for instance, what 
   can be said about the category of locally finite dimensional $\mathcal{D}$-modules, and  for each primitive ideal $P$ of 
   $\mathcal{D}$ can one find  a canonical irreducible module (hopefully a factor of a Verma module) whose annihilator equals $P$?
   The first step in these questions is to classify the primitive ideals of $\mathcal{D}$, as we do here. 
 \medskip
 
\noindent{\bf 1.4. Dixmier-Moeglin equivalence.} The validity of the Dixmier-Moeglin Equivalence  for an algebra $R$
  yields simultaneous representation-theoretic, algebraic and topological characterisations of the primitive ideals amongst the 
  prime ideals of $R$.  This equivalence  is a feature of some but not all  
 Hopf $k$-algebras; see, for example,  \cites{Be, BL} for discussions of when it holds for a noetherian (Hopf) algebra.  As  we prove in 
 Theorem~\ref{main2},  $\mathcal{D}$ satisfies the equivalence. See $\S$\ref{DM} for 
the   details, where we also give a  rather ambitious Conjecture ~\ref{wild}, proposing a general result 
  encompassing all affine Noetherian pointed Hopf $\mathbb{C}$-algebras. 

\medskip

\noindent {\bf Notation.} Throughout,   all vector spaces and all
 unadorned tensor products are understood to be over the base field $k$. We denote the comultiplication of a Hopf algebra
 $H$ by $\Delta$ and  its augmentation ideal by $H^+$. The Gelfand-Kirillov, or GK dimension  of an object $X$  is denoted by 
 $\GKdim(X)$, while  the global (homological) dimension, respectively injective dimension of a ring $R$ is denoted by 
 $\mathrm{gldim}(R)$, respectively  $\mathrm{injdim}(R)$. For precision, we specify that in the \emph{Ore extension} 
 $T=S[ v;\sigma,\partial]$, multiplication
   is defined    by 
\begin{equation}\label{Ore-eq}
vs= s^\sigma v+\partial(s) \qquad\text{for} \ s\in S.
\end{equation}
It follows that $\partial$ is a $\sigma$-derivation in the sense that $\partial(ab)=a^\sigma \partial(b)+\partial(a)b.$ This follows
 the conventions of, for example,  \cite[p.34]{GW}.

\bigskip

\section{Preliminaries}\label{define}
 
\subsection{Definitions and notation}\label{notation}  The following definitions and notation from \cite{ADPP}) 
will remain in play throughout the paper. First, the \emph{Jordan plane} is  
$$  J\;  := \; k\langle x,y : [y,x] = -\hhalf x^2 \rangle, $$
with \emph{bosonization}
$$ \mathcal{H} \; := \; J\#C_{\infty} = J\#\langle g^{\pm 1} \rangle, \textit{ where } gxg^{-1} = x, \; gyg^{-1} = y + x. $$
Then the \emph{Drinfeld double} of $J$ is defined to be $ \mathcal{D} :=   \mathcal{H}\langle u,v, \zeta \rangle,$ with additional relations as follows:
\begin{align*} [u,v]&= \hhalf u^2; \; \;[\zeta,v] = -v; \; \;[\zeta, u] = -u; \; [u,y] = 1-g;\\
[v,x] &= 1-g + xu; \;\; [v,y] = yu - g\zeta; \; [v,g] = gu; \;\; [\zeta,y] = y; \;\; [\zeta, x] = x; \; \\
[x,u] &= [x,g] = [u,g] =[\zeta,g] = 0.
\end{align*}
The coalgebra structure, which will mostly not concern us here, is determined for $\mathcal{H}$ by specifying that $g$ is grouplike 
and $x$ and $y$ are $(g,1)-$primitive; and then extended to $\mathcal{D}$ by setting $u$ and $\zeta$ to be primitive and  
$\Delta (v) = v \otimes 1 + 1 \otimes v + \zeta \otimes u.$

\medskip

Observe that $K := k\langle u,v,\zeta \rangle$ is a Hopf subalgebra of $\mathcal{D}$ and in fact, as one can see from the PBW theorem
 for $\mathcal{D}$ as described in \cite[Proposition~2.3(ii)]{AP},   $\mathcal{D} = J\otimes_k K$ as vector spaces. As noted in \cite[Lemma~2.2]{AP} there
  is a non-degenerate skew pairing between $J$ and $K$ which yields the multiplication relations between these subalgebras as 
  in \cite{DT}.

\medskip 
\subsection{Initial results}\label{initial}
We gather together in Theorem~\ref{before} some of the main results of \cite{AP} and \cite{ADPP}. We must first define  some 
elements of $\mathcal{D}$, as follows. Set
\begin{equation}\label{s-q-defn}
 q \; := \;  ux + 2(1 + g), \; \textit{ and }\; s \; := \; xv + uy + (-\frac{1}{2}ux + g - 1)\zeta - 2(g + 1). 
 \end{equation}
The following lemma is partly explicit, partly implicit, in \cite[$\S$4]{ADPP}. Given a $k$-algebra automorphism $\sigma$ of a 
$k$-algebra $H$, we say that the element $h$ of $H$ is $\sigma$-\emph{normal} if $ha = \sigma(a)h$ for all $a \in H$. 

\begin{lemma}\label{normal} Keep the above notation.
\begin{enumerate}
\item[(i)] $q$ and $s$ are both $\sigma$-normal, where $\sigma$ is the automorphism of $\mathcal{D}$ defined by 
$$ \sigma (y) = y +\hhalf x, \; \; \sigma (v) = v -\hhalf u,$$
with $\sigma$ acting as the identity on the other generators of $\mathcal{D}$.
\item[(ii)] $\sigma^2$ equals conjugation by $g$ on $\mathcal{D}$; that is, $ \sigma^2 (h) \; = \; ghg^{-1}$ for all $h \in \mathcal{D}$.
\item[(iii)] The elements $ z \; := \; q^2 g^{-1}, \;  \theta \; := \; s^2g^{-1}$  and  $\omega \; := \; qsg^{-1}$
are in the centre $Z(\mathcal{D})$ of $\mathcal{D}$. 
\end{enumerate}
\end{lemma}
\begin{proof} (i) and (ii) are easy checks, and (iii) is immediate from (i) and (ii).
\end{proof}

\begin{theorem}\label{before} Retain the notation introduced above.
\begin{enumerate}
\item[(i)]\cite[Proposition~2.7(i)]{AP} $\mathcal{O}(G) := k\langle x,u,g^{\pm 1} \rangle$ is a normal commutative Hopf subalgebra of $\mathcal{D}$, with $G = ((k,+)\times (k,+))\rtimes (k^{\ast}, \times)$.

\item[(ii)]\cite[Proposition~2.7(ii)]{AP} $\mathcal{D}\mathcal{O}(G)^+$ is a Hopf ideal of $\mathcal{D}$, with an isomorphism of Hopf algebras 
\begin{equation}\label{epi} \mathcal{D}/\mathcal{D O}(G)^+ \; \cong \; U(\mathfrak{sl}_2(k)).
\end{equation}
\item[(iii)]\cite[Theorem~3.11]{ADPP} The finite dimensional irreducible $\mathcal{D}$-modules are the finite dimensional irreducible $U(\mathfrak{sl}_2(k))$-modules given by the epimorphism (\ref{epi}).
 \item[(iv)]\cite[Theorem~4.10]{ADPP} With the notation from Lemma~\ref{normal}, the centre of $\mathcal{D}$ is
\begin{equation}\label{centre} Z(\mathcal{D}) \; = \; k \langle z, \omega, \theta \, : \,  z\theta = \omega^2 \rangle, \end{equation}
 \item[(v)]\cite[Remark~2.2]{ADPP} $\mathcal{D}$ is pointed. 
\qed
\end{enumerate}
\end{theorem} 
 
\medskip

We'll need the following labelling of the maximal ideals of $Z(\mathcal{D})$. Note that here  there are two maximal ideals of $Z(\mathcal{D})$ which require particular attention.

\begin{notation}\label{maxspecZ} (i)  By Theorem~\ref{before}((iv), $\mathrm{Maxspec}(Z(\mathcal{D}))$ consists of 
$$ \{\mathfrak{m}_{(\alpha,\gamma)} := \langle z - \alpha^2 \gamma^{-1}, \omega - \alpha, \theta - \gamma \rangle: \alpha \in k,\, \gamma \in k^{\ast}\} \;\dot\cup \;\{\mathfrak{m}_{\beta} := \langle z-\beta, \omega, \theta \rangle : \beta \in k \}.
$$
Note that  $\mathfrak{m}_{(\alpha,\gamma)}$ can be simplified to  $\mathfrak{m}_{(\alpha,\gamma)} = \langle  \omega - \alpha, \theta - \gamma \rangle$, while  $\mathfrak{m}_{\beta} = \langle z- \beta, \omega \rangle$ when $\beta \neq 0$. 

\noindent (ii) It is easy to calculate using the definition of the counit that
$$ \mathcal{D}^+ \cap Z(\mathcal{D})\; = \;\mathcal{O}(G)^+\mathcal{D} \cap Z(\mathcal{D})\; = \; \mathfrak{m}_{(-16,16)}.$$
We thus denote $\mathfrak{m}_{(-16,16)}$ by $\mathfrak{m}_+$.

\noindent (iii) It is clear that the singular locus of $Z(\mathcal{D})$ is $\{\mathfrak{m}_0 \}$.
\end{notation}

\bigskip

\section{Ring-theoretic preparations}\label{prep}

In this section we assemble some properties needed in the analysis of the primitive spectrum of $\mathcal{D}$. The proofs are most easily approached by viewing $\mathcal{D}$ as an iterated Hopf Ore extension starting not from the base field $k$ but from the commutative normal Hopf subalgebra $\mathcal{O}(G) = k\langle x,u, g^{\pm 1} \rangle$ of Theorem~\ref{before}(i). More precisely:

\begin{proposition}\label{prop-ore} $\mathcal{D}$ is an iterated Ore extension
\begin{equation}\label{IHOE} \mathcal{D} \; = \; \mathcal{O}(G)[y;\delta_1][\zeta; \delta_2][v; \tau, \delta_3],
\end{equation}
where the derivations $\delta_1$ and $\delta_2$, the automorphism $\tau$ and the $\tau$-derivation $\delta_3$ can be read off from the defining relations of $\mathcal{D}$ given in Subsection~\ref{notation}. 

 In particular, $\mathcal{D}$ is a noetherian domain.
\end{proposition}

\begin{proof}  Use the proof of \cite[Proposition~1.6]{AP} to show  that $\mathcal{D}$ has basis $$\{g^ax^bu^cy^d\zeta^ev^f : a\in \mathbb{Z}, b,\dots,f\in\mathbb{N}\}.$$ Then the form of  the    relations \eqref{notation} combined with
\cite[Theorem~1, p.438]{Cohn}  show that it is indeed  an Ore extension. 
\end{proof} 

Although it will be not needed in this paper, the description \eqref{IHOE} even describes $\mathcal{D}$ as an Iterated Hopf Ore Extension (IHOE), in the sense that each extension in that formula is itself a Hopf algebra. It also  shows that, by setting 
\begin{align*}\label{degree} &\deg x = \deg u = \deg g = \deg g^{-1} = 0;\\
&\deg y = \deg \zeta = \deg v =1, 
\end{align*}
one obtains  a filtration $\mathcal{F}$ on $\mathcal{D}$ with associated graded algebra 
\begin{equation}\label{grade} \gr_{\mathcal{F}}\mathcal{D} \; = \; \mathcal{O}(G)[\overline{y}, \overline{\zeta},\overline{v}].
\end{equation}
So $\gr_{\mathcal{F}}\mathcal{D}$ is a commutative polynomial algebra in 6 variables with one variable inverted.

\medskip

\subsection{Homological properties}\label{homology}
 In this subsection we note that $\mathcal{D}$ has certain useful homological properties, and we begin with the relevant definitions.    
A ring $A$ is called  \emph{Auslander Gorenstein} if it has finite injective dimension and satisfies the
 {\it Gorenstein condition}: if $p<q$ 
are non-negative integers and  $M$ is a finitely generated $A$-module, then
$\Ext_A^p(N,\,A)=0$ for every submodule $N$ of 
$\Ext_A^q(M,\,A)$. The ring $A$ is \emph{Auslander regular} if it is Auslander Gorenstein of finite global dimension.
Set
 $j_A(M)= \min\{ r : \Ext^r_A(M,A)\not= 0\}$ for the \emph{homological grade} of $M$.  
 Then an Auslander Gorenstein ring $A$ of finite GK dimension is
 called \emph{GK-Cohen-Macaulay} (or just CM), provided 
 that  $j_A(M)+\GKdim(M)=\GKdim(A)$ holds
for each such  $M$.  Obviously affine commutative regular rings are both Auslander regular and CM. 

\begin{proposition}\label{homprop}\begin{enumerate} 
 \item[(i)] $\mathcal{D}$ is Auslander regular and CM.
\item[(ii)] $\mathcal{D}$ is AS regular in the sense of, say,  \cite{Lev}.
\item[(iii)]  $\GKdim(\mathcal{D}) = 6 =  \mathrm{gldim}(\mathcal{D}) .$
\item[(iv)] GK dimension is an exact function on finitely generated $\mathcal{D}$-modules.
\end{enumerate}
\end{proposition}
 
\begin{proof} 
(i) By \cite[Remark, p.157]{Bj} the filtration $\mathcal{F}$  is Zariskian and so the result follows  from \cite[Theorems~3.8, 3.9 and Remark, p. 165]{Bj}.

\medskip
\noindent(ii) This is immediate from (i) and \cite[Lemma~6.1]{BZ}.

 \medskip 
 \noindent   (iii)   By \cite[Corollary~1.4]{McCSt},   $\GKdim(\mathcal{D}) = \GKdim(\gr_{\mathcal{F}}(\mathcal{D})) = 6 .$ 
 
 \smallskip
       By Proposition~\ref{prop-ore} and   \cite[Theorem~7.5.3(i)]{McCR}, we have $\mathrm{gldim}(\mathcal{D}) \leq 6$.
    By \cite[Theorem~3.11]{ADPP}  $\mathcal{D}$ has a finite dimensional module, say $M$ and  the CM condition implies that  $M$ has homological dimension $\geq 6$. Hence $\mathrm{gldim}(\mathcal{D}) = 6$.

\medskip

\noindent(iv) Since  $j_{\mathcal{D}}$ is exact on finitely generated $\mathcal{D}$-modules by  \cite[Theorem~2.3]{Lev}, this follows from the CM condition.
\end{proof}

\subsection{Key lemma}\label{key}

The following lemma will be crucial in our analysis of the primitive spectrum of $\mathcal{D}$. In its proof, given an ideal B of a noetherian ring $S$, we denote by $\sqrt B$ the ideal of $S$ such that $\sqrt B/B$ is the nilradical of $S/B$.

\begin{lemma}\label{small} Let $M$ be a finitely generated (right or left) $\mathcal{D}$-module such that either $\mathrm{Ann}_{k[x]}(M) \neq 0$ or $\mathrm{Ann}_{k[u]}(M) \neq 0$. Then 
\begin{enumerate}
\item[(i)] there exists $r \geq 1$ such that 
$$ (\mathfrak{m}_+\mathcal{D})^r \subseteq (\mathcal{O}(G)^+\mathcal{D})^r \subseteq \mathrm{Ann}_{\mathcal{D}}(M); $$
\item[(ii)] $\GKdim(M) \leq 3$.
\end{enumerate}
\end{lemma}

\begin{proof} (i) Let $I := \mathrm{Ann}_{\mathcal{D}}(M)$, an ideal of $\mathcal{D}$. Assume that $I \cap k[x] \neq 0$, the proof in the other case being exactly similar, but with $k\langle u,v \rangle$ replacing $J$. One easily confirms that every non-zero prime ideal of the Jordan plane $J = k\langle x, y\rangle$ contains $x$. Therefore, since $I \cap k[x] \neq 0$, there exists $N \geq 1$ such that 
$x^N \in I \cap k[x]  \subseteq I \cap \mathcal{O}(G).$ Since $\mathcal{O}(G)$ is commutative,
\begin{equation}\label{gotcha1}x \in \sqrt(I \cap \mathcal{O}(G)). 
\end{equation}
Since $I$ is an ideal of $\mathcal{D}$, $[v,\,I] \subseteq I$; moreover, from the defining relations of $\mathcal{D}$ and the fact that $\mathcal{O}(G) = k\langle x,u,g^{\pm 1} \rangle$,
$[v, \,\mathcal{O}(G)] \subseteq \mathcal{O}(G)$. Therefore
\begin{equation}\label{caught1} [v , \, I \cap \mathcal{O}(G)] \subseteq I \cap \mathcal{O}(G).
\end{equation}
Since $k$ has characteristic 0 it follows from (\ref{caught1}) and \cite[Lemma~3.20]{GW} that
\begin{equation}\label{caught3} [v,\, \sqrt (I \cap \mathcal{O}(G))] \subseteq \sqrt (I \cap \mathcal{O}(G)).
\end{equation}
By (\ref{gotcha1}) and (\ref{caught3})
$$ [v, \, x] = 1-g + xu \in \sqrt (I  \cap \mathcal{O}(G)), $$
so that $(1 - g) \in \sqrt (I  \cap \mathcal{O}(G))$. Then 
$$ [v, \, g - 1] = [v,g] = gu \in \sqrt (I  \cap \mathcal{O}(G)), $$
so that  $u \in \sqrt (I  \cap \mathcal{O}(G))$. Since $\mathcal{O}(G)^+$ is generated by $x$, $u$ and $g - 1$ we deduce that $\mathcal{O}(G)^+ \mathcal{D} \subseteq \sqrt I$, proving (i).

\medskip 

\noindent (ii) By (i) $M$ is a finitely generated $\mathcal{D}/( \mathcal{O}(G)^+\mathcal{D})^r$-module for some $r \geq 1$. Since $\mathcal{D}/\mathcal{O}(G)^+\mathcal{D} \cong U(\mathfrak{sl}(2,k)$ by Theorem~\ref{before}(ii), and so has GK dimension 3, (ii) follows from this and Proposition~\ref{homprop}(iv).
\end{proof}

\bigskip

\bigskip

\subsection{Ore localisations of $\mathcal{D}$}\label{local} 
 
To help in the analysis of its primitive spectrum we need four  Ore localisations of $\mathcal{D}$. The first of these is described in 
\cite[Theorem~4.8]{ADPP}, and the others are similar. These sets are described as follows:

\begin{definition}\label{Ore} Label the following four subsets of $\mathcal{D}$: 
\begin{align*} A \; &:= \; \{q^i : i\geq 0 \} \, \dot\cup \, \{ x^j : j \geq 0 \}, \; \; \; B \; := \; \{s^i : i\geq 0 \} \, \dot\cup \, \{ x^j : j \geq 0 \} ,\\
C \; &:= \; \{q^i : i\geq 0 \} \, \dot\cup \, \{ u^j : j \geq 0 \}, \; \; \; D \; := \; \{s^i : i\geq 0 \} \, \dot\cup \, \{ u^j : j \geq 0 \}.
\end{align*} \end{definition}
 
\begin{lemma}\label{adnilp} {\rm (1)} The elements $x$ and $u$ act ad-locally-nilpotently on $\mathcal{D}$. Consequently, $\{x^i : i \geq 0 \}$ and $\{u^i : i \geq 0 \}$ are Ore sets in $\mathcal{D}$.

{\rm (2)} For each $\Omega\in \{A,B,C,D\}$ the set $\Omega$ is an Ore set of regular elements of $\mathcal{D}$, and we write the corresponding localisation as $\mathcal{D}_{(\Omega)}$.
\end{lemma}

\begin{proof}  (1)   For $x$ this is proved in  \cite[Lemma~4.3(i)]{ADPP}. The claim for  $u$ is a similar easy consequence of the defining relations of $\mathcal{D}$.

(2) Localising at the powers of $q$ is the same as localising at the powers of $q^2$ or even at the powers
$z=q^2g^{-1}$, since $g$ is a unit. Thus, for $\Omega=A$ or $\Omega=C$ and appealing to Lemma~\ref{normal}(iii),  we can 
replace $q$ by the central element $z$. Similarly in the other two cases we can replace $s$ by the central element $\theta$. Thus in each case we wish to  localise at  one central and one locally ad-nilpotent element in the domain $\mathcal{D}$.  Thus it is indeed    an Ore set of regular elements. \end{proof}

Thus each of the four rings  $\mathcal{D}_{(\Omega)}$ is a subalgebra of the quotient division algebra $Q(\mathcal{D})$ of $\mathcal{D}$ that contains   $\mathcal{D}$. As we next show, each of these rings is a localisation of the  second Weyl algebra over a commutative ring.

\begin{notation}\label{Weyl} 
{\rm (i)}  In $\mathcal{D}_{(A)}$, set $p_A\, := \, -2q^{-1}x^{-1}y, \;\;q_A \, := \,q, \; \; t_A \, := \, qx^{-1},$ and $\eta_A \, := \, -xq^{-1}\zeta.$
 \begin{enumerate}
\item[(ii)] In $\mathcal{D}_{(B)}$, set $p_B\, := \, -2s^{-1}x^{-1}y, \;\; q_B \, := \, s, \;\; t_B \, := \, sx^{-1}.\;\; \eta_B\, := \, -xs^{-1}\zeta$.

\item[(iii)] In $\mathcal{D}_{(C)}$, set $p_C\, := \, 2q^{-1}u^{-1}v, \;\;q_C \, := \, q, \;\; t_C \, := \, q^{-1}u^{-1}, \;\; \eta_C \, := \, -uq\zeta$.

\item[(iv)] In $\mathcal{D}_{(D)}$, set $p_D\, := \, 2s^{-1}u^{-1}v, \;\; q_D \, := \, s, \;\; t_D\, := \, s^{-1}u^{-1}, \;\;  \eta_D \, := \, us\zeta.$
\end{enumerate}

We further set $z_{\Omega}:=z$ for $\Omega=A$ and $\Omega=C$ but $z_{\Omega} := \theta$ when $\Omega = B, D$.
\end{notation}

The motivation behind the above definitions becomes clear from the following lemma. For $\Omega=A$, this was obtained in the proof of \cite[Theorem~4.8]{ADPP}. The claims regarding the other elements can be checked by a similar direct calculation.

\begin{lemma}\label{relations} Let  $\Omega\in \{A,B,C,D\}.$
    Then  we have  the following relations in $Q(\mathcal{D})$:
$$ [p_\Omega, q_\Omega] \, = \, 1 \, = [\eta_\Omega,t_\Omega], \, \textit{ with all other brackets being zero;} \qed$$
\end{lemma}

When $\Omega=A$ the following result is  given in \cite[Theorem~4.8]{ADPP}, although we give a proof that works for all 4 cases simultaneously.

\begin{theorem}\label{Weyls} For each $\Omega\in \{A,B,C,D\}$, the localisation $\mathcal{D}_{(\Omega)}$ is a localised Weyl algebra over its centre. More precisely:
  $$  \mathcal{D}_{(\Omega)} \; = \;  A_2^{(\Omega)}(k) \otimes S_{(\Omega)}, $$
where $A_2^{(\Omega)}(k)$ denotes the localisation of the second Weyl algebra over $k$ with generators $p_\Omega, q_\Omega^{\pm 1}, \eta_\Omega, t_\Omega^{\pm 1}$, while $S_{(\Omega)} $ is the commutative ring 
$S_{(\Omega)}= k[z_\Omega^{\pm 1}, \omega]$. 
\end{theorem}

\begin{proof}  The generators $z, \, \omega$ and $\theta$ of $Z(\mathcal{D})$ are given in Lemma~\ref{normal}(iii),  from which 
it follows that the subalgebra $S_{(\Omega)}$    of $Q(\mathcal{D})$ is  contained   in the centre $Z(\mathcal{D}_{(\Omega)})$.
 Therefore we can consider the subalgebra 
\begin{equation}\label{into} \mathcal{E}_{(\Omega)} \; := \; S_{(\Omega)}\langle p_\Omega,q_\Omega,t_\Omega,\eta_\Omega \rangle \ \subseteq \  \mathcal{D}_{(\Omega)}.
\end{equation} 
We claim that
the  inclusion (\ref{into}) is an equality.  
In order to prove this, check that   given  generators of  $\mathcal{D}_{\Omega}$ are 
contained in $ \mathcal{E}_{(\Omega)}$. Thus, for example, when  $\Omega=A$, one shows that
 $ \{q^{-1}, x^{\pm 1}, y, \zeta, g^{\pm 1}, u,v \} \subset \mathcal{E}_{(A)},$ and similarly in the other cases. Thus, 
 $\mathcal{E}_{(\Omega)}  =  \mathcal{D}_{(\Omega)}, $ as claimed.
As noted in the proof of Lemma~\ref{adnilp} 
the localisation of $\mathcal{D}$ at $\Omega$   involves inverting one  central and one  ad-nilpotent element of $\mathcal{D}$.  Thus, by Proposition~\ref{homprop}(iii) and 
  \cite[Lemma~4.7]{KL},  $\GKdim(D_{(\Omega})=\GKdim( \mathcal{D})=6$. We conclude that 
  $  \GKdim (\mathcal{E}_{(\Omega)})  =\GKdim (\mathcal{D})=6.$

On the other hand,  by Lemma~\ref{relations} $\mathcal{E}_{(\Omega)}$ is a factor of the  ring 
\begin{equation*}\label{whole} V_{(\Omega)} \, := \, S_{(\Omega)} \otimes_k A_2^{(\Omega)}(k),\end{equation*}  
which is also a domain of GK-dimension 6. So if $\mathcal{E}_{(\Omega)}$ were a proper factor of   $V_{(\Omega)},$ then 
\cite[Corollary~8.3.6]{McCR} would  imply that
$\GKdim(\mathcal{E}_{(\Omega)}) <  6, $  giving a contradiction.

So the only  possibility is that 
$ \mathcal{E}_{(\Omega)} \; \cong \; V_{(\Omega)} = S_{(\Omega)} \otimes_k A_2^{(\Omega)}(k),$
as required.  \end{proof}

\bigskip

\section{The primitive spectrum of $\mathcal{D}$}\label{privspec}

In this section we describe the primitive spectrum of $\mathcal{D}$. This splits naturally into several cases:
\begin{itemize}
\item the primitive ideals  not containing  $ \mathfrak{m}_+$ or $  \mathfrak{m}_0$; these are the generic ones;
\item the  ideal $ \mathfrak{m}_+\mathcal{D}$, which is also primitive;
\item  the  ideal $  \mathfrak{m}_0\mathcal{D}$, for which $\sqrt{\mathfrak{m}_0\mathcal{D}}$ is a unique prime ideal $P_0$; 
\item finally, $P_0$ is also maximal.
\end{itemize}
The details are given in the next four subsections with  the results being  combined in Subsection~\ref{privshape}.

In this section  and in Section~\ref{primesec} we will without further reference  use of the yoga for prime ideals of Noetherian rings under Ore localisation as described in, for example, \cite[Theorems~10.18 and  10.20]{GW}. We use Notation~\ref{maxspecZ} to describe the maximal ideals of $Z(\mathcal{D})$ and Definition~\ref{Ore} to define Ore sets in $\mathcal{D}$.

\subsection{The generic minimal primitives}\label{generic}

We begin by looking at the generic case.

\begin{theorem}\label{genpriv} Let $\mathfrak{m}$ be a maximal ideal of $Z(\mathcal{D})$ with $\mathfrak{m} \neq \mathfrak{m}_+$ and $\mathfrak{m} \neq \mathfrak{m}_0$. Then the following are true.
\begin{enumerate}
\item[(i)] $\mathfrak{m}\mathcal{D}$ is a completely prime maximal ideal of $\mathcal{D}$.
\item[(ii)] The localisation of $\mathcal{D}/\mathfrak{m}\mathcal{D}$ at the powers of (the image of) either $x$ or   $u$ is isomorphic to a localised Weyl algebra $A^{(\Omega)}_2(k)$,  where $\Omega\in\{A,B,C,D\}$.
\item[(iii)] $\GKdim(\mathcal{D}/\mathfrak{m}\mathcal{D}) = 4$.
\item[(iv)] $\mathfrak{m}\mathcal{D}$ is generated by a central regular sequence of length 2.
\item[(v)] $\mathcal{D}/\mathfrak{m}\mathcal{D}$ is CM and is Auslander Gorenstein  with $\mathrm{injdim}(\mathcal{D}/\mathfrak{m}\mathcal{D})< 4.$
\end{enumerate}
\end{theorem}

\begin{proof} (i), (ii)   
By Notation~\ref{maxspecZ},  $\mathfrak{m} = \langle z - \alpha, \omega - \beta, \theta - \gamma \rangle$ with $\alpha, \beta, \gamma \in k$ and $\alpha\gamma = \beta^2$. Moreover, thanks to the hypothesis on $\mathfrak{m}$, either $(a)$ $\alpha \neq 0$ or $(b)$ $\gamma \neq 0$.

Assume $(a)$. We prove (ii) for the localisation at the powers of $x$. (The arguments for powers of $u$ are exactly similar, but using the Ore sets $C$ and $D$ rather than $A$ and $B$.) Using the notation of $\S$\ref{local} and applying Theorem~\ref{Weyls}, we see that $\mathfrak{m}\mathcal{D}_{(A)}$ is a maximal ideal of 
$\mathcal{D}_{(A)}$. Observe that, since $A := \{q^i, x^j : i,j \geq 0 \}$ and
$$ z = q^2 g^{-1} \equiv \alpha \neq 0 \, \textit{mod}(\mathfrak{m}\mathcal{D}), $$
$A^{(A)}_2(k)$ is isomorphic to the localisation of $\mathcal{D}/\mathfrak{m}\mathcal{D}$ at the powers of $x$. Define
$$ P_{\mathfrak{m}} \, := \, \mathfrak{m}\mathcal{D}_{(A)} \cap \mathcal{D},$$
so that $P_{\mathfrak{m}}$ is a completely prime ideal of $\mathcal{D}$ with $\mathfrak{m}\mathcal{D} \subseteq P_{\mathfrak{m}}$. By definition of $P_{\mathfrak{m}}$,
\begin{equation}\label{cover} \mathfrak{m}\mathcal{D}_{(A)} \; = \; P_{\mathfrak{m}}\mathcal{D}_{(A)}. 
\end{equation}

We claim that in fact
\begin{equation}\label{nogap} P_{\mathfrak{m}} \; = \; \mathfrak{m}\mathcal{D}.
\end{equation}
Since $\mathcal{D}$ is (left) noetherian there exist $e_1, \ldots , e_t \in P_{\mathfrak{m}}$ such that $P_{\mathfrak{m}} = \mathfrak{m}\mathcal{D} + \sum_{i=1}^t \mathcal{D}e_i$. By (\ref{cover}), for each $i = 1, \ldots , t$ there exist $f_i \in \mathfrak{m}\mathcal{D}$ and $s_i \in \mathbb{Z}_{\geq 0}$ such that 
\begin{equation}\label{shot} e_i \; = \; f_i x^{-s_i}. 
\end{equation}
Define  $ s := \mathrm{max}\{s_i : 1 \leq i \leq t \} \in \mathbb{Z}_{\geq 0}$, and  
$$ I \; := \; \{ \tau \in \mathcal{D} : P_{\mathfrak{m}}\tau \subseteq \mathfrak{m}\mathcal{D} \}. $$
Thus $I$ is an ideal of $\mathcal{D}$ containing $\mathfrak{m}\mathcal{D}$ and, by (\ref{shot}), 
$ x^s \in I. $
If $s = 0$ then $I = \mathcal{D}$; otherwise we see from Lemma~\ref{small} that $(\mathfrak{m}_+)^r \subset I $ for some $r \geq 1$. Since also $\mathfrak{m} \subseteq I$ and $\mathfrak{m} \neq \mathfrak{m}_+$ by hypothesis, it follows that $I = \mathcal{D}$, and (\ref{nogap}) is proved. 

In case $(a)$ it remains only to prove that $P_{\mathfrak{m}}$ is a maximal ideal of $\mathcal{D}$. Suppose then that $J$ is an ideal of $\mathcal{D}$ with $P_{\mathfrak{m}} \subsetneq  J$. Then $J\mathcal{D}_{(A)} = D_{(A)}$ by the maximality of the ideal $P_{\mathfrak{m}}\mathcal{D}_{(A)}$ of $\mathcal{D}_{(A)}$. Again using the fact that $q + \mathfrak{m}\mathcal{D}$ is a unit of $\mathcal{D}/\mathfrak{m}\mathcal{D}$ we see that $x^s \in J$ for some $s \geq 1$. Then, as before, Lemma~\ref{small} implies that $J = \mathcal{D}$.

\medskip
Suppose that $(b)$ holds rather than $(a)$. Then the element $s$ is a unit $\textit{mod}\, \mathfrak{m}\mathcal{D}$, so we use the same argument as for $(a)$, but working with $\mathcal{D}_{(C)}$ rather than $\mathcal{D}_{(A)}$.

\medskip

(iii) By (ii) and \cite[Example~3.7 and Theorem~4.9]{KL} the localisation of $\mathcal{D}/\mathfrak{m}\mathcal{D}$ at the powers of $x$ has GK dimension 4. Since $\mathrm{ad}(x)$ acts nilpotently on $\mathcal{D}/\mathfrak{m}\mathcal{D}$ by Lemma~\ref{adnilp}, it follows from \cite[Theorem~4.9]{KL} that $\GKdim(\mathcal{D}/\mathfrak{m}\mathcal{D}) \; = \; 4.$

\medskip

(iv) Again we assume $(a)$ that $z - \alpha \in \mathfrak{m}$ for $\alpha \in k \setminus \{0\}$, the proof in case $(b)$ being similar. We can begin a regular central sequence in $\mathfrak{m}\mathcal{D}$ with $z - \alpha$. Since $\mathcal{D}$ is CM of GK-dimension 6 by Proposition~\ref{homprop}$(i,iii)$, it follows from \cite[Theorem~7.2(b)]{GL} that $\mathcal{D}/(z - \alpha)\mathcal{D}$ is CM of GK-dimension 5. Moreover, by \cite[Remark~2.4]{Lev} the CM property ensures that $\mathcal{D}/(z - \alpha)\mathcal{D}$ is GK-homogeneous; that is, it contains no non-zero ideal with GK-dimension strictly less than 5. Since $Z(\mathcal{D})/(z - \alpha)Z(\mathcal{D})$ is a polynomial algebra we can choose $y \in \mathfrak{m}$ such that $\mathfrak{m} = \langle z - \alpha, y \rangle$. If $y + (z - \alpha)\mathcal{D}$ is a zero divisor in $\mathcal{D}/(z - \alpha)\mathcal{D}$ we obtain a non-zero ideal of  $\mathcal{D}/(z - \alpha)\mathcal{D}$ killed by $\mathfrak{m}\mathcal{D}$, contradicting the GK-homogeneity of  $\mathcal{D}/(z - \alpha)\mathcal{D}$ in view of (iii). Thus $\{z - \alpha, y \}$ is a regular central sequence in $\mathfrak{m}\mathcal{D}$. 

\medskip
(v) Since $\mathcal{D}$ is CM by Proposition~\ref{homprop}(i), $R=\mathcal{D}/\mathfrak{m}\mathcal{D}$ is CM with $\GKdim(R)= 4$ by (iv) and two applications of \cite[Theorem~7.2(b)]{GL}. The Auslander Gorenstein property is given by (iv) and \cite[$\S$3.4, Remark (3)]{Lev}. As $R$ is simple it cannot have a finite dimensional module. Hence   $\mathrm{injdim}(R)<4$ follows from the next lemma. 
\end{proof}

The following observation is well-known.
\begin{lemma}\label{trivial} Let $R$ be a noetherian, Auslander Gorenstein, CM ring and write  $\GKdim(R)=m$. Then $\mathrm{injdim}(R)\leq m$. Moreover  $\mathrm{injdim}(R)= m\iff $ $R$ has a finite dimensional representation.
\end{lemma}

\begin{proof} Let $n=\mathrm{injdim}(R)$ and pick a finitely generated $R$-module $M$ such that $\Ext^n_R(M,R)\not=0$. By
the  Auslander condition and the spectral sequence \cite[Theorem~2.2]{Lev} $j(E^{nn}(M))=n$ for $E^{nn}=\Ext^n(\Ext^n(M,R),R)$. 
By the CM property $\GKdim(E^{nn}(M)) = m-n$ and the result follows easily.
\end{proof}

\bigskip

\subsection{Non-generic minimal primitives (I) - $\mathfrak{m}_+$}\label{plus}
The next case to  consider is  $\mathfrak{m}\mathcal{D}$ for $\mathfrak{m} =\mathfrak{m}_+$, as we do here.  Recall from 
Notation~\ref{maxspecZ}(ii) that $\mathfrak{m}^+ = \mathcal{D}^+ \cap Z(\mathcal{D}) = \langle \omega+16, \theta - 16 \rangle$.

We start with a subsidiary result, which works for any field $k$ of characteristic zero.

\begin{theorem}\label{Nullstel}
$\mathcal{D}$ is a Jacobson ring that satisfies the Nullstellensatz, in other words:
\begin{enumerate}
\item[(i)] every prime ideal of $\mathcal{D}$ is an intersection of primitive ideals;
\item[(ii)] for every simple $\mathcal{D}$-module $M$,  $\mathrm{End}_{\mathcal{D}}(M)$ is algebraic over $k$. In particular, every primitive ideal of $\mathcal{D}$ contains a maximal ideal of  $Z(\mathcal{D})$.
\end{enumerate}
\end{theorem}

\begin{proof} By \eqref{grade}, $\mathcal{D}$ has a filtration $\mathcal{F}$ such that the associated graded ring $\gr_{\mathcal{F}}(\mathcal{D})$ is a commutative affine ring. Hence by \cite[Corollary~1.7]{McCSt} there is a second filtration $\mathcal{G}$ by finite dimensional $k$-subspaces of $\mathcal{D}$ such that 
$\gr_{\mathcal{F}}(\mathcal{D})$ is also a commutative and  affine ring.  The result now follows from \cite[Theorem~0.4]{ASZ}.
\end{proof}

\begin{theorem}\label{plusthm}
\begin{enumerate}
\item[(i)] $\mathfrak{m}_+\mathcal{D}$ is a completely prime, primitive ideal of $\mathcal{D}$.
\item[(ii)] The localisation of $\mathcal{D}/\mathfrak{m}_+\mathcal{D}$ at the powers of $x$ or the powers of $u$ is a localisation of the Weyl algebra $A_2 (k)$ at powers of a generator.
\item[(iii)] $\mathfrak{m}_+\mathcal{D}$ is generated by a regular central sequence of length 2.
\item[(iv)] $\mathcal{D}/\mathfrak{m}_+\mathcal{D}$ is Auslander  Gorenstein and  CM with  
$$\GKdim (\mathcal{D}/\mathfrak{m}_+\mathcal{D}) = 4 = \mathrm{injdim}(\mathcal{D}/\mathfrak{m}_+\mathcal{D}).$$
\item[(v)] Every prime  ideal $P$ of $\mathcal{D}$ which strictly contains $\mathfrak{m}_+\mathcal{D}$ satisfies 
$$ \mathcal{O}(G)^+\mathcal{D} \; \subseteq \; P,$$
so the space of such primes $P$ is homeomorphic to $\mathrm{Spec}(U(\mathfrak{sl}(2,k))$.
\end{enumerate}
\end{theorem}

\begin{proof} Recall that $q_A=q$. Since $q^2 \equiv 16g \not\equiv 0 \,\textit{mod}(\mathfrak{m}_+\mathcal{D}),$ Theorem~\ref{Weyls} implies that $\mathfrak{m}_+\mathcal{D}_{(A)}$ is a maximal ideal of $\mathcal{D}_{(A)}$, with 
$\mathcal{D}/\mathfrak{m}_+\mathcal{D}_{(A)}\cong A^{(A)}_2(k)$. Therefore, defining 
$P_+ := \mathfrak{m}_+\mathcal{D}_{(A)} \cap \mathcal{D}$, we deduce that $P_+$ is a completely prime ideal of $\mathcal{D}$ with 
\begin{equation}\label{hurt} \mathfrak{m}_+\mathcal{D} \; \subseteq \; P_+.
\end{equation}
We will eventually show that (\ref{hurt}) is an equality. 

As in the proofs of Theorem~\ref{genpriv}(i),(ii), let $I$ be the right annihilator in $\mathcal{D}$ of $P_+/\mathfrak{m}_+\mathcal{D}$. Then $I$ contains $\mathfrak{m}_+\mathcal{D}$ and a power of $x$, and hence, by Lemma~\ref{small}, 
\begin{equation}\label{trap} (\mathcal{O}(G)^+\mathcal{D})^r \; \subseteq \; I \qquad \text{for some $r \in \mathbb{Z}_{\geq 1}$},
\end{equation}
In particular, $\GKdim(\mathcal{D}/I) \leq 3$ by Lemma~\ref{small}(ii). Therefore, by \cite[Proposition~5.1(d)]{KL}
\begin{equation}\label{small2} \GKdim(P_+/\mathfrak{m}_+\mathcal{D}) \leq 3.
\end{equation}
Recall that $\GKdim(A^{(A)}_2(k)) = 4$ by \cite[Example~7.3 and Theorem~4.9]{KL}, so that 
\begin{equation}\label{trend} \GKdim(\mathcal{D}/P_+) \; = \; 4
\end{equation}
by \cite[Theorem~4.9]{KL}. Thus, from (\ref{small2}), (\ref{trend}) and Proposition~\ref{homprop}(iv) it follows that
\begin{equation}\label{corner} \GKdim(\mathcal{D}/\mathfrak{m}_+\mathcal{D}) \; = \; 4.
\end{equation}

By Proposition~\ref{homprop},  $\mathcal{D}$ is CM  and Auslander regular, with 
$ \mathrm{gldim}(\mathcal{D}) = 6= \GKdim(\mathcal{D}).$ 
It therefore follows from the CM property of $\mathcal{D}$
together with (\ref{corner}) 
 that
\begin{equation}\label{canon} j_{\mathcal{D}}(\mathcal{D}/\mathfrak{m}_+\mathcal{D}) \; = \; 6-4 \; = \; 2.
\end{equation}
From (\ref{canon}) and \cite[Proposition~3.6]{B} we deduce that the maximum length of a regular sequence of elements of $\mathfrak{m}_+$ on $\mathcal{D}$ is precisely 2; in particular any choice of a generating pair of elements of $\mathfrak{m}_+$, for example, $\{z-16, \, \omega + 16\}$, is a regular sequence on $\mathcal{D}$. Therefore, by two applications of \cite[Theorem~7.2(b)]{GL}, 
\begin{equation}\label{rad}\mathcal{D}/\mathfrak{m}_+\mathcal{D} \textit{ is CM of GK-dimension~4.}
\end{equation}
Similarly, two applications of \cite[$\S$3.4, Remark~(3)]{Lev} show that $\mathcal{D}/\mathfrak{m}_+\mathcal{D}$ is Auslander Gorenstein.   By Lemma~\ref{small}(i) and Theorem~\ref{before}(ii),  
$U(\mathfrak{sl}(2,k))\cong \mathcal{D}/\mathcal{D}\mathcal{O}(G)^+$ is a factor of  $\mathcal{D}/\mathfrak{m}_+\mathcal{D}$
and so  $\mathcal{D}/\mathfrak{m}_+\mathcal{D}$ has a non-zero finite dimensional module, $M$.
Thus, by Lemma~\ref{trivial}, $\mathrm{injdim}(\mathcal{D}/\mathfrak{m}_+\mathcal{D})=4.$

By \cite[Remark~2.4]{Lev}, again, the CM property  for $\mathcal{D}/\mathfrak{m}_+\mathcal{D}$ implies that  $\mathcal{D}/\mathfrak{m}_+\mathcal{D}$ is GK-homogeneous. Therefore we may conclude from (\ref{small2}) that $P_+$ does indeed equal $\mathfrak{m}_+\mathcal{D}$. This proves (i) - (iv), with the exception of  showing that $\mathfrak{m}_+\mathcal{D}$ is primitive.

\medskip

(v) Let $Q$ be a prime ideal of $\mathcal{D}$ with $\mathfrak{m}_+\mathcal{D} \subsetneq Q$. 
 As already noted, $q$ is congruent to a unit $\textit{mod }\mathfrak{m}_+\mathcal{D}$. Then $Q\mathcal{D}_{(A)} = \mathcal{D}_{(A)}$ by $(ii)$, so  $Q$ must contain a power of $x$. Hence, by Lemma~\ref{small},  $\mathcal{O}(G)^+\mathcal{D} \subseteq Q$, as required.

\medskip

Finally, to see that $\mathfrak{m}_+\mathcal{D}$ is primitive note that (v) shows that it is locally closed.  Hence it is primitive by Theorem~\ref{Nullstel}(i).
\end{proof}

\bigskip

\subsection{Non-generic minimal primitives (II) - $\mathfrak{m}_0$}\label{zero}
In this subsection we begin our study of the ideal $\mathfrak{m}_0\mathcal{D}$.
Recall the definition of $q,s$ from \eqref{s-q-defn} and,  from Notation~\ref{maxspecZ}(iii),  that $ \mathfrak{m}_0  :=  \langle q^2g^{-1}, \, qsg^{-1}, \, s^2g^{-1} \rangle $ is the unique singular point of $\mathrm{Maxspec}(Z(\mathcal{D}))$. Clearly the right ideal
\begin{equation}\label{Pdef} P_0 \; := \; q\mathcal{D} + s\mathcal{D} 
\end{equation}
is a two-sided ideal of $\mathcal{D}$ since $q$ and $s$ are normal in $\mathcal{D}$ by Lemma~\ref{normal}. Moreover, 
$$ \mathfrak{m}_0\mathcal{D} = P_0^2 \subset P_0 \subseteq \sqrt{\mathfrak{m}_0\mathcal{D}}.$$
As part of the next proposition we see that $P_0$ is completely prime, so the second inclusion above is an equality. In fact $P_0$ is a maximal ideal, but this is more difficult to prove, and is delayed until $\S$\ref{steps}.

\begin{proposition}\label{singular}Retain the above notation, and set $T := \mathcal{D}/P_0$. 
\begin{enumerate}
\item[(i)] $T$ is a localisation of a 4-step iterated Ore extension of $k$, namely
$$ T \; = \; \big((k[u,x]\langle (ux + 2)^{-1}\rangle)[y; \partial_1]\big)[v; \sigma, \partial_2],$$
where $u$ and $x$ commute, 
$$\partial_1(u) = -\hhalf ux - 2, \quad \partial_1 (x) = -\hhalf x^2,$$
$$ \partial_2(u) = - \hhalf u^2, \quad \partial_2 (x) = \frac{3}{2}ux +2, \quad \partial_2 (y) = \frac{3}{2}uy - 2,$$
and $\sigma (y) = y +\hhalf x$, with $\sigma (x) = x$ and $\sigma (u) = u.$

\item[(ii)] $\{q,s\}$ forms a regular normal sequence of generators of $P_0$.

\item[(iii)] $ \mathrm{gldim}(T) \leq 4 = \GKdim(T).$
\item[(iv)] $T$ is CM  and is an Auslander regular domain.
\end{enumerate}
\end{proposition}

\begin{proof}Throughout the proof we abuse notation by simply denoting the image in $T$ of an element $\omega$ of $\mathcal{D}$ by $\omega$ when no confusion seems likely. 

(i),(ii) Since $q := ux + 2(1 +g)$ and $q \equiv 0\,\textit{mod}(P_0)$, we can write 
\begin{equation}\label{up} g \equiv - \hhalf(ux + 2)\,\textit{mod}(P_0),
\end{equation}
so that 
\begin{equation}\label{gee} ux+ 2 \textit{ is a unit in }T.
\end{equation}
Using (\ref{up}) we find that, $\textit{mod}(P_0)$,
$$ s \; := \;  xv + uy + (- \hhalf ux + g - 1)\zeta - 2(g + 1) \;\equiv  \; xv + uy + 2g\zeta - 2g - 2,$$ 
so that,  since $s\in P_0$,
\begin{equation}\label{greek} \zeta \; \equiv \; - \hhalf g^{-1}(ux + xv + uy)\quad mod(P_0)
\end{equation}
It follows from (\ref{up}), (\ref{gee}) and (\ref{greek}) that
\begin{equation}\label{gener}
T \; = \; k\langle u,x, (ux+ 2)^{-1}, y,v \rangle. 
\end{equation}
The relations for $\mathcal{D}$ given in $\S$\ref{notation} immediately imply the following relations for  the generators for $T$ listed  in (\ref{gener})  
\begin{align*}\label{Trels} [u,x] &= 0, \qquad &[y,x] = - \hhalf x^2, \qquad [v,x] &= {\textstyle\frac{3}{2}}ux + 2,\\
[y,u] &= - \hhalf ux - 2, \qquad &[v,u] = -\hhalf u^2, \qquad [v,y] &= \textstyle{\frac{3}{2}}uy + \hhalf xv - 2.
\end{align*}
Clearly the iterated Ore extension of $k[u,x]\langle(ux + 2)^{-1}\rangle$ defined in (i), which we temporarily label $\widehat{T}$, satisfies precisely these relations, so there is an algebra epimorphism $\Phi$ from $\widehat{T}$ onto $T$.  

We next show that  $\Phi$ is an isomorphism, which we do by computing $\GKdim(T)$. First note that $\GKdim(\widehat{T}) = 4$ by \cite[Theorem~12.3.1]{KL}, since it is a PBW extension in 2 variables of $k[u,x]\langle(ux + 2)^{-1}\rangle$.  Thus,  certainly
$\GKdim(T) \leq 4.$ On the other hand $\mathcal{D}$ is CM of GK-dimension 6 by Proposition~\ref{homprop}($i,iii$). Hence, because $q$ is a regular normal element of $\mathcal{D}$ by Lemma~\ref{normal}, $\mathcal{D}/q\mathcal{D}$ is CM of GK-dimension 5 by \cite[Theorem~7.2(b) and its proof]{GL}. Moreover $\mathcal{D}/q\mathcal{D}$ is GK-homogeneous by \cite[Remark~2.4]{Lev}. Since $\GKdim(T) \leq 4$, this ensures that 
\begin{equation}\label{sequence} s \textit{  cannot be a zero-divisor mod}\,\, q\mathcal{D}.
\end{equation}
Since $P_0 \; := \; q\mathcal{D} + s\mathcal{D}$, 
 a second application of \cite[Theorem~7.2(b) and its  proof]{GL} yields 
$ \GKdim(T)  = 4$ and also shows that 
\begin{equation}\label{four} T \textit{ is CM.}
\end{equation}  Since $\widehat{T}$ is a domain, the equality $\GKdim(\widehat{T}) = 4 =\GKdim(T)$, combined with \cite[Proposition~3.15]{KL}, shows that  $ \widehat{T} =T$.  Thus (i) is proved, with (ii) also following thanks to (\ref{sequence}).

\medskip

\noindent (iii)  
By (i),  $T$ is a 2-step iterated Ore extension of $k[u.x]\langle(ux + 2)^{-1}\rangle$,  and so  two applications of \cite[Theorem~7.5.3(i)]{McCR} gives   $\mathrm{gldim}(T) \leq 4$.

\medskip

\noindent (iv) That $T$ is a domain is clear from (i), while the CM property was proved in \eqref{four}.
 The Auslander Gorenstein property   holds for $\mathcal{D}$ by Proposition~\ref{homprop}(i). Thus,  by  (ii) and  two applications of \cite[Theorem~7.2(a)]{GL},  $T$ is also  Auslander Gorenstein and it is then Auslander regular by (iii). \end{proof}

We remark that, by Lemma~\ref{trivial} and Theorem~\ref{before}(iii)  it follows that $\mathrm{gldim}(T)<4$. We do not know 
the exact value of $\mathrm{gldim}(T)$.  
\bigskip

\subsection{Maximality of $P_0$}\label{steps}

Let $T := \mathcal{D}/P_0$ as in Proposition 5.3. Define also the following subalgebras of $T$:
$$ R \; := k\langle u,x, (ux +2)^{-1}\rangle, \textit{  and  } S \; := \; R[y;  \partial_1],$$
so that $T = S[v;\sigma,  \partial_2]$.   It is important to note that, by the formul\ae\ in Proposition~\ref{singular},  $R$ is preserved by the $\sigma$-derivation $\partial_2$. Moreover, since $\sigma_{|R}$ is the identity, $\partial_2$ actually restricts to a derivation on $R$.

 It is much easier to determine when an Ore extension is simple if the ring is a differential operator ring, in the sense that the  defining automorphism is actually the identity. Thus we will reduce to that case. The idea follows from Lemma~\ref{normal} which shows that $\sigma^2$ is  given by the inner automorphism $\tau_g$ in the sense that $\sigma(s)=\tau_g(s)=gsg^{-1}$ for suitable $g\in S$. We will therefore extend $R, S$ and $ T$ by   $\sqrt{g}$ and show that $\sigma$ is then inner, and so can be removed. The details are given in the next few results, culminating in Proposition~\ref{prop-diffop}.
 
 \begin{notation}\label{h-notation} In the algebraic closure of $R$, set $h=(ux+2)^{-\hhalf}$. Write
  $\Rtilde=R\langle h\rangle=k\langle u,x,h,h^{-1}\rangle$. We extend the $\partial_i$ to derivations on $\Rtilde$ by the usual rules for fractional powers:
  $$\partial(h)= (-\half)(ux+2)^{-1}h\partial(ux+2),$$
  for $\partial= \partial_1,\partial_2$. Set $\Stilde = \Rtilde[y; \partial_1]$. Finally, we can extend $\sigma$ to $\Rtilde$ and $\Stilde $ by setting $\sigma(h)=h$. Then both $\sigma$ and $\partial_2$ are naturally defined on $\Stilde$ as an automorphism, respectively $\sigma$-derivation and so $\Ttilde=\Stilde[v;\sigma,\partial_2]$ is a well-defined Ore extension of $\Stilde$.
 
 \end{notation}
 
 The following observation will prove useful.
 
 \begin{lemma}\label{freedom} $\Stilde$ is a free left and right $S$ module on basis $\{1,h\}$. Similarly, $\Ttilde$ is a free left and right $T$ module on basis $\{1,h\}$.
 \end{lemma}
 
 \begin{proof} As $h^2=(ux+2)^{-1}\in R$, the construction of $\Rtilde$ ensures that $\Rtilde$  is a free left and right $R$-module  on basis $\{1,h\}$.
 We can then write 
 $$\Stilde= \bigoplus_{i=0}^\infty \Rtilde y^i = \bigoplus Ry^i \oplus\bigoplus Rh y^i=\bigoplus Ry^i \oplus \bigoplus R y^i h.$$
 Collecting terms shows that $\Stilde = S\oplus Sh$. As $S$ is a domain this is necessarily a direct sum of free modules. The same argument works for $\Ttilde$.
 \end{proof}

\begin{lemma}\label{lemma-inner} On $\Stilde$, $\sigma$ is the inner automorphism $\tau_{h^{-1}}$; thus   $\sigma(f)=h^{-1}fh$  for $f\in \Stilde$.
\end{lemma}

\begin{proof} Since $\Rtilde$ is a commutative ring on which $\sigma$ is the identity, the lemma holds trivially on $\Rtilde$. It therefore just remains to check that the automorphisms agree on $y$. To prove this, we  rewrite $h^{-1}yh$  as follows.

\[\begin{aligned}
h^{-1}yh \ = & \ \   (ux+2)^{\half} y (ux+2)^{-\half} \\
= & \ \  (ux+2)^{\half} (ux+2)^{-\half}y \ + \  (ux+2)^{\half}\cdot\partial_1\big( (ux+2)^{-\half}\big)\\
= & \ \  y \ +\ (ux+2)^{\half} (-\hhalf )(ux+2)^{-\frac{3}{2}}\cdot\partial_1( (ux+2))\\
= & \ \  y \ -\ \hhalf(ux+2)^{-1} \Big( (-\hhalf ux-2) x - u(\hhalf x^2) \Big)\\
= & \ \  y \ -\ \hhalf(ux+2)^{-1} \bigl(-(ux+2)x\bigr)\\
= & \ \  y \ + \hhalf x \ = \ \sigma(y);
\end{aligned}\]
as required.
\end{proof}

\begin{proposition}\label{prop-diffop} Set $\alpha=hv$. Then $\Ttilde$ is the Ore extension $\Ttilde=\Stilde[\alpha; \partialtilde]$
where $\partialtilde$ is the derivation of $\Stilde$ defined by 
$\partialtilde(s)=h\partial_2(s)$ for $s\in \Stilde$; thus 
\[\partialtilde(u)=-\hhalf hu^2,\quad \partialtilde(x) = h({\textstyle\frac{3}{2}}ux+2) \quad\text{and}\quad \partialtilde(y) = h(({\textstyle\frac{3}{2}}uy-2).\]
As such, $\Ttilde$ is a noetherian domain.
\end{proposition}

\begin{proof} This is a formal computation. Indeed, for $s\in \Stilde$, Lemma~\ref{lemma-inner} implies that $\sigma(s)=h^{-1}sh$. 
Equivalently,
\begin{equation}\label{commutingalpha}
\begin{aligned}
\alpha s \ = & \ hv s  = h\sigma(s)v + h\partial_2(s) \\
\ = & \ hh^{-1}shv + h\partial_2(s) = s\alpha + h\partial_2(s).
\end{aligned}\end{equation}
Therefore,  since $\Ttilde=\Stilde[v;\sigma,\partial_2]=\bigoplus \Stilde v^i$, we see that $\Ttilde = \bigoplus \Stilde\alpha^i$. Since $\Ttilde$ is a domain, combining this with \eqref{commutingalpha} and \cite[Theorem~1, p.438]{Cohn} gives the desired conclusion.
\end{proof}

Our next aim will be to show that the ring $\Ttilde$ is a simple domain, after which it is easy to prove the same conclusion for $T$. We start with some preparatory results.

\begin{lemma}\label{lemma-primality}
If there exists a non-zero $(\partial_1,\partialtilde)$-invariant ideal $I$ in $\Rtilde$, then there exists a non-zero   $(\partial_1,\partialtilde)$-invariant prime  ideal $P$ in $\Rtilde$.
\end{lemma}

\begin{proof} Using \cite[Lemma~3.18(b)]{GW}  twice, clearly $I\Stilde$ is a proper non-zero   ideal of $\Stilde $ and then  $I\Ttilde$ is a proper nonzero 
ideal of $\Ttilde $. Pick a prime ideal $Q\supseteq I\Ttilde$. Then, by \cite[Lemmata~3.18 and~3.21]{GW}, twice, $Q_1=Q\cap \Stilde$
 is a $\partialtilde$-invariant prime ideal of $\Stilde$ and hence $Q_2=Q_1\cap \Rtilde $ is a $\partial_1$-invariant prime ideal of $\Rtilde$. However, since $\Rtilde$ and $Q_1$  are both  $\partialtilde$-invariant, so is $Q_2$. Thus, $P=Q_2$ is the desired prime ideal.
\end{proof}

\begin{proposition}\label{prop-invariant ideals}
There is no proper,  non-zero $(\partial_1,\partialtilde)$-invariant ideal $I$ in $\Rtilde$.
\end{proposition}

\begin{proof}  Suppose that there exists such an ideal $I$. By Lemma~\ref{lemma-primality} we can and will assume that $I$ is a prime ideal.
Suppose, first, that  $(xu+\lambda)\in I$, for some  $\lambda\in k$. Then
\[I \ \ni \ \partial_1(xu+\lambda) \ = \  (-\hhalf ux-2)x-(\hhalf x^2)u \ =\  -(ux+2)x\]
and 
\[I\ \ni \  \partialtilde(xu+\lambda) \ = \ h\bigl(-\hhalf u^2x + ({\textstyle\frac{3}{2}}ux+2)u\bigr)
 \ = \ h\bigl(ux^2+2u\bigr) \ = \ h(ux+2)u.\]
 As $(xu+2)^{-1}=h^2\in \Rtilde$, clearly $\lambda\not=2$ and so the two equations imply that $x\in I$, respectively $u\in I$. Thus, $I=x\Rtilde + u\Rtilde$.  But, now $I\ni \partial_1(u)=-\hhalf ux-2$ and so $I=\Rtilde$, a contradiction. We conclude that 
 \begin{equation}\label{eq-C-defn}
 I\cap\mathcal{C}=\emptyset\quad\text{for } \mathcal{C}=\{(xu+\lambda): \lambda \in k^*\}
 \end{equation}
 Since $I$ is a prime ideal it follows that $\mathcal{C}\subseteq \mathcal{C}(I)$ and hence that $I_{\mathcal{C}}$ is a proper prime ideal of the localisation $\Rtilde_{\mathcal{C}}$. 
 
 Next, if $I_{\mathcal{C}}\ni f=f(u)$ for some $f(u)\in k[u]$, then $I_{\mathcal{C}}\ni \partial_1(f)= -\hhalf(ux+4)\frac{df}{du}$. Hence $\frac{df}{du}\in I_{\mathcal{C}}$. By induction on $ \deg f$, this implies that $I_{\mathcal{C}} = \Rtilde_{\mathcal{C}}$, a contradiction. Thus $I_{\mathcal{C}}\cap k[u]^{\ast}=\emptyset$ and so we can further localise at $\mathcal{S}=k[u]^{\ast}$ and conclude that $I_{\mathcal{C}\mathcal{S}}$ is a proper prime ideal of $\Rtilde_{\mathcal{C}\mathcal{S}}$.
 Now consider $\Rtilde_{\mathcal{C}\mathcal{S}}$. We have $\Rtilde=k\langle u,x,h,h^{-1}\rangle$ and $h^{-2}=(ux+2)$ whence $x=u^{-1}(h^{-2}-2)$. Thus $\Rtilde_{\mathcal{C}\mathcal{S}}=\Rtilde_{\mathcal{S}\mathcal{C}}$ is a localisation of $k(u)[h,h^{-1}]$. 
 
 The advantage of working in $\Rtilde_{\mathcal{C}\mathcal{S}}$ is that we can simplify our derivation $\partialtilde$. On $\Rtilde$ and 
 $\Rtilde_{\mathcal{C}\mathcal{S}}$ write $\partial_u=\frac{\partial}{\partial u}$ and $\partial_x=\frac{\partial}{\partial x}$. Then, as derivations on either   ring,
 \[\partial_1=-(\hhalf xu +2)\partial_u \ - \ \hhalf x^2\partial_x\ \]
while %\ \quad\text{while}\quad\
\[\partialtilde \ = \  -\hhalf hu^2\partial_u \ + \ h({\textstyle\frac{3}{2}}ux+2)\partial_x.\]

 We now set $\mu := -hu^2(ux+4)^{-1}$ and  take 
 \[\partialtilde' \ :=  \ \partialtilde+\mu\partial_1=\Bigl(-\hhalf hu^2 + \mu(-\hhalf ux -2)\Bigr)\partial_u \ + \ 
 \Bigl(h({\textstyle\frac{3}{2}} ux + 2) -\hhalf x^2\mu\Bigr)\partial_x.\]
    This element $\mu$ has been chosen so that the coefficient of $\partial_u$  in $\partialtilde'$  is 
 \[-\hhalf(ux+4)^{-1}\Bigl(  hu^2(ux+4) - (ux+4)hu^2\bigr) \ = \ 0.\]
 Therefore,
 \[ \begin{aligned} 
 \partialtilde' \ = \ & \Bigl( h({\textstyle\frac{3}{2}} ux + 2) \ + \ \hhalf h x^2 u^2(ux+4)^{-1}\Bigr) \partial_x \\
  \ = \ &  (ux+4)^{-1} h \Bigl( ({\textstyle\frac{3}{2}} ux + 2)(ux+4) + \hhalf x^2u^2 \Bigr) \partial_x\\
  \ = \ & (ux+4)^{-1} h \Bigl( 2u^2x^2 + 8ux + 8\Bigr) \partial_x \\
   \ = \ &     \beta \partial_x   \qquad \textit{for}\quad \beta \ := \ 2(xu+4)^{-1}(ux+2)^2h.\\
 \end{aligned}\]
 
 Since $I_{\mathcal{C}\mathcal{S}}$ is invariant under both $\partial_1$ and $\partialtilde$, it is also invariant under $\partialtilde'$. Since $\beta$ is a unit in $\Rtilde_{\mathcal{C}\mathcal{S}}$, it follows that 
 \begin{equation}\label{eq-partialx}
 I_{\mathcal{C}\mathcal{S}}   \textit{ is also invariant under } \beta^{-1}\partialtilde' = \partial_x.
 \end{equation}
 
Thus, by (\ref{eq-partialx}) and the expression given above for $\partial_1$, $I_{\mathcal{C}\mathcal{S}}$ is invariant under $ (\hhalf ux+2)\partial_u$, and therefore under $\partial_u$ since $\hhalf ux + 2$ is a unit. So $I_{\mathcal{C}\mathcal{S}}$ is invariant under $\partial_u$ and $\partial_x$. Since $\Rtilde_{\mathcal{C}\mathcal{S}}$ is a localisation of $k[u,x]$ this  forces $I_{\mathcal{C}\mathcal{S}} = \Rtilde_{\mathcal{C}\mathcal{S}}$, giving the required contradiction. 
\end{proof}

In order to pass between $T$ and $\Ttilde$ we need:

\begin{lemma}\label{lemma-swap} If $\Ttilde$ is a simple ring then so is $T$.
\end{lemma}

\begin{proof}  Suppose that $T$ has a proper ideal $J$. Then $X=\Ttilde/J\Ttilde$ is a $(T,\Ttilde)$-bimodule. Moreover,  by Lemma~\ref{freedom} $\Ttilde$ is a finitely generated left $T$-module and so  $X$ is a finitely generated left $T$-module; say $X=\sum_{i=1}^r T x_i$. Then, as $\Ttilde$ is an Ore domain, $\text{ann}_{\Ttilde}(X) = \bigcap_i \text{ann}_{\Ttilde}(x_i)\not=0$. Since $\Ttilde $ is a simple ring this implies that $\text{ann}_{\Ttilde}(X)=\Ttilde$ and hence that $X=0$. In other words, $J\Ttilde = \Ttilde$.  

On the other hand,  by Lemma~\ref{freedom},  $\Ttilde= T+T h$ is a free left $T $-module  and so 
$J\Ttilde = J\oplus Jh\not=\Ttilde$. This contradiction proves the lemma.\end{proof}

We   now put everything together and prove the main result of this subsection.

\begin{theorem}\label{simplicity} $T$ is a simple ring.
\end{theorem}

\begin{proof} By Lemma~\ref{lemma-swap} it suffices to prove that $\Ttilde$ is simple. By \cite[Theorem~1.8.4]{McCR} applied to $\Ttilde=\Stilde[\alpha; \partialtilde]$, we need to prove 
\begin{enumerate} 
\item $\partialtilde$ is not an inner derivation on $\Stilde$, and 

\item $\Stilde$ has no proper $\partialtilde$-invariant ideals.
 \end{enumerate}
 
 Now, as $\partial_1(x)=-\hhalf x^2$, the right  ideal $x\Stilde$  is a proper  two-sided ideal of $\Stilde$. As such, it is preserved by any inner derivation of $\Stilde$. But $\partialtilde(x) = h({\textstyle\frac{3}{2}}ux + 2) \not\in x\Stilde$, this means $\partialtilde$ cannot be an inner derivation of $\Stilde$ and so (a) holds.
 
 Suppose that $\Stilde$ has a proper $\partialtilde$-invariant ideal $I$. Then, by \cite[Lemma~3.18]{GW}, $K=I\cap \Rtilde$ is a $\partial_1$-invariant ideal of $\Rtilde$, while by \cite[Lemma~3.19]{GW}, $K\not=0$. Since both $I$ and $\Rtilde$ are both $\partialtilde$-invariant, so is $K$. In other words, $K $ is a proper $(\partial_1,\partialtilde)$-invariant ideal of $\Rtilde$. This contradicts  Proposition~\ref{prop-invariant ideals}.  Thus (b) holds and so  \cite[Theorem~1.8.4]{McCR}  implies that $\Ttilde$ is simple.
 \end{proof}

\begin{remark}\label{birational-remark}
We end the subsection by noting that $\Ttilde$ is obviously birational to  the Weyl algebra $A_2$. We do not know if the same is true for $T$ itself.
\end{remark}

\bigskip
\subsection{The shape of the primitive spectrum of $\mathcal{D}$}\label{privshape}

In this subsection we combine the earlier results of this section to prove Theorem~\ref{main1}.
By Theorem~\ref{Nullstel},  every primitive ideal $P$ of $\mathcal{D}$  contains a maximal ideal of $Z(\mathcal{D})$. Thus $\mathrm{Privspec}(\mathcal{D})$ is the disjoint union

\begin{equation}\label{split} 
\mathrm{Privspec}(\mathcal{D}) \; = \; \dot\bigcup_{\mathfrak{m} \in \mathrm{Maxspec}(Z(\mathcal{D}))}\mathcal{V}(\mathfrak{m}) 
\end{equation}
where $ \mathcal{V}(\mathfrak{m})=
 \{ P \in \mathrm{Privspec}(\mathcal{D}) : \mathfrak{m} \subseteq P \}.$
 There are thus 3 cases, corresponding to $\S\S$\ref{generic}, \ref{plus} and \ref{zero}.

\medskip

\noindent {\bf (I)} $\mathcal{V}(\mathfrak{m})$, where  $\mathfrak{m} \in \mathrm{Maxspec}(Z(\mathcal{D}))$ with  $\mathfrak{m} \neq \mathfrak{m}_+$ and $\mathfrak{m} \neq \mathfrak{m}_0$.  By Theorem~\ref{genpriv}, $\mathcal{V}(\mathfrak{m}) = \{\mathfrak{m}\mathcal{D}\}$ is a single generic maximal ideal of $\mathcal{D}$. Moreover $\mathcal{D}/\mathfrak{m}\mathcal{D}$ is birationally equivalent to the second Weyl algebra, with other properties as listed in that theorem.

\medskip

\noindent {\bf (II)} $\mathcal{V}(\mathfrak{m}_+)$. By  Theorem~\ref{plusthm}, this consists of $\mathfrak{m}_+\mathcal{D}$, together with
$$\mathcal{V}(\mathcal{O}(G)^+\mathcal{D}) \; := \; \{P \in \mathrm{Privspec}(\mathcal{D}) : \mathcal{O}(G)^+\mathcal{D} \subset P\},$$ 
which is homeomorphic to $\mathrm{Privspec}(U(\mathfrak{sl}(2,k)))$ by Theorem~\ref{before}(ii). 

  Recall that $\mathrm{Privspec}(U(\mathfrak{sl}(2,k)))$ is composed of the co-Artinian maximal ideals $\{M_n \,:\, n \in \mathbb{Z}_{\geq 1}\}$, where $M_n = \mathrm{Ann}(V_n)$, $V_n$ being the $n$-dimensional irreducible $U(\mathfrak{sl}_2(k))-$module, together with the \emph{minimal primitives} of $U(\mathfrak{sl}(2,k))$; that is, the ideals $( \Omega - \lambda)U(\mathfrak{sl}(2,k))\, : \, \lambda \in k\}$, where   $\Omega$ is the Casimir element. Each $M_n$ contains one such minimal primitive and each minimal primitive is contained in at most one $M_n$; the remaining minimal primitives are also maximal. Note that $\mathcal{O}(G)^+\mathcal{D}$ is prime but not primitive since $\mathcal{D}/\mathcal{O}(G)^+\mathcal{D} \cong U(\mathfrak{sl}(2,k))$ and this domain satisfies the Nullstellensatz and has non-trivial centre $k[\Omega]$.

\medskip 

\noindent {\bf (III)}   $\mathcal{V}(\mathfrak{m}_0)$. This  is the singleton  $  \{P_0= q\mathcal{D} + s\mathcal{D}=\sqrt{\mathfrak{m}_0}\}$, by Proposition~\ref{singular} and Theorem~\ref{simplicity}. \bigskip

\section{Prime ideals and the Dixmier-Moeglin equivalence}\label{primeDM}
 In this section we prove Theorem~\ref{main2}  from the introduction, which describes the prime ideals of $\mathcal{D}$,  and we discuss the Dixmier-Moeglin equivalence for $\mathcal{D}$. 
 
\subsection{The prime spectrum of $\mathcal{D}$}\label{primesec}
We need the following lemmas for the proof of the main result,  Theorem~\ref{prime}.

\begin{lemma}\label{intersect} Let $P$ be a nonzero prime ideal of $\mathcal{D}$. Then $P \cap Z(\mathcal{D}) \neq \{0\}$. 
\end{lemma}
\begin{proof} If $x^i \in P$ for some $i \geq 0$ then $\mathcal{O}(G)^+\mathcal{D} \subseteq P$ by Lemma~\ref{small} applied with $M = \mathcal{D}/P$, and therefore $\mathfrak{m}^+ = \mathcal{O}(G)^+\mathcal{D} \cap Z(\mathcal{D}) \subseteq P,$ proving the lemma for $P$. So we may assume that $\{x^i : i \geq 0 \} \cap P = \emptyset$. Similarly, we may assume that $\{q^j : j \geq 0 \} \cap P = \emptyset$, since otherwise $ 0 \neq q^ng^{-2n} \in P \cap Z(\mathcal{D})$ for some $n \geq 0$ and again the result follows for $P$.

Hence, using  Notation~\ref{Weyl} and Theorem~\ref{Weyls}, $P\mathcal{D}_{(A)}$ survives as a non-zero proper ideal of
$\mathcal{D}_{(A)}= \mathcal{D}\langle q^{-1}, x^{-1} \rangle \; = \; A_2^{(A)}(k) \otimes_k S_{(A)},$
where $A_2^{(A)}(k)$ is a localised Weyl algebra and $S_{(A)} = k[z^{\pm 1}, \omega]$. In particular,
\begin{equation}\label{control} P\mathcal{D}_{(A)} \; = \; (P\mathcal{D}_{(A)} \cap S_{(A)})\mathcal{D}_{(A)}.
\end{equation}
By \cite[Theorem~10.20]{GW} and the discussion in the first paragraph of this proof, $P = P\mathcal{D}_{(A)} \cap \mathcal{D}$, and therefore
\begin{equation}\label{caught} P \cap Z(\mathcal{D})  \; = \;  P\mathcal{D}_{(A)} \cap Z(\mathcal{D})\; = \;(P\mathcal{D}_{(A)} \cap S_{(A)}) \cap Z(\mathcal{D}).
\end{equation} 
Since the $Z(\mathcal{D})$-module $S_{(A)}/Z(\mathcal{D})$ is $\{z^i\}$-torsion, that is $\{q^{i} g^{-2i}\}$-torsion, it follows from (\ref{control}) and (\ref{caught}) that $P \cap Z(\mathcal{D}) \neq \{0\}$ as required.
\end{proof}

\medskip

Note that, since $k$ is algebraically closed of characteristic 0, the defining relation $z\theta = \omega^2$ of $Z(\mathcal{D})$ can be rewritten using a linear change of variables as the quadratic form $X^2 + Y^2 = Z^2$. Thus a proof of the next result can be found at \cite[p.51 and Proposition~11.4]{F}. 

\begin{lemma}\label{class} All height one primes of $Z(\mathcal{D})$ are principal except $\mathfrak{p}_1 := \langle z,\omega \rangle$ and $\mathfrak{p}_2 := \langle \theta, \omega \rangle$.\qed
\end{lemma}

Here is the main result of this section, using in (ii)  the notation of Lemma~\ref{class}. This proves Theorem~\ref{main2} from the introduction.

\begin{theorem}\label{prime} Let $P$ be a prime but not primitive ideal of $\mathcal{D}$. 
\begin{itemize}
\item[(i)] There are the following three possibilities for $P$.
\begin{enumerate}
\item[(a)] $P = \{0\}$.
\item[(b)] $P = \mathcal{O}(G)^+\mathcal{D}$.
\item[(c)] $P$ has height one and is minimal over $(P \cap Z(\mathcal{D}))\mathcal{D}$ for a height one prime ideal  $P \cap Z(D)$ of $Z(\mathcal{D})$.
\end{enumerate}
\item[(ii)] In case ($c$), if $P \cap Z(\mathcal{D}) = \mathfrak{p}_i$ for $i = 1$, resp. $i = 2$, then $P = q\mathcal{D}$, resp. $P = s\mathcal{D}$. The remaining primes in case ($c$) are precisely the set 
$$\{ P \;: P \; = \; f\mathcal{D} \},$$
as $f$ ranges through the equivalence classes of irreducible elements of $Z(\mathcal{D})$ other than the associates of $z,\omega, \theta$.
\end{itemize}  
\end{theorem}

\begin{proof} Note first that $\{0\}$ is completely prime by Proposition~\ref{prop-ore}, and is not primitive, because $\mathcal{D}$ satisfies the   Nullstellensatz by Theorem~\ref{Nullstel} and $ Z(\mathcal{D})\not=k$. This covers case ($a$).

Let $P$ be a non-zero prime but not primitive ideal of $\mathcal{D}$. By Lemma~\ref{intersect},
$$ \{0\} \; \neq \; \mathfrak{p} \; := \; P \cap Z(\mathcal{D}).$$
If $\mathfrak{p} = \mathfrak{m}_+$ then Theorem~\ref{plusthm} together with the discussion at $\S$\ref{privshape}({\bf II}) shows that the only possibility is $P = \mathcal{O}(G)^+\mathcal{D}$, which is completely prime but is again not primitive thanks to the Nullstellensatz, since $ Z(U(\mathfrak{sl}(2,k)))\not=k$. This is case ($b$).

 If $\mathfrak{p} = \mathfrak{m}_0$ then $P = P_0$, which is maximal by Theorem~\ref{simplicity}, so this case can't happen. Similarly,  $\mathfrak{p}$ is any maximal ideal of $Z(\mathcal{D})$ apart from $\mathfrak{m}_+$ or $\mathfrak{m}_0$, then $P=\mathfrak{p}\mathcal{D}$ is a maximal ideal of $\mathcal{D}$ by Theorem~\ref{genpriv}(i),  which  again gives a contradiction.

So we are left with the case when $\mathfrak{p}$ has height one. Assume first that $\mathfrak{p} = fZ(\mathcal{D})$ is principal. Then, by Lemma~\ref{class}, $z = q^2g^{-1} \notin P$, and $\{x^i : i \geq 0\} \cap P = \emptyset$ by Lemma~\ref{key} Therefore, using Notation~\ref{Weyl} and Theorem~\ref{Weyls} 
$$  \mathfrak{p}\mathcal{D}_{(A)} \; =\;  (P \cap S_{(A)})\mathcal{D}_{(A)}\; = \; P\mathcal{D}_{(A)}.$$
We claim that $P = \mathfrak{p}\mathcal{D}$. To see this, note that $\mathfrak{p}\mathcal{D} = f\mathcal{D}$ is principal, so that $\mathcal{D}/ \mathfrak{p}\mathcal{D}$ is CM of GK-dimension 5, by \cite[Theorem~7.2(b) and its proof]{GL}, and GK-homogeneous by \cite[$\S$3.4, Remark~(3)]{Lev}.  Now  $P/\mathfrak{p}\mathcal{D} $ it is killed by $\mathfrak{p}\mathcal{D}$ and by a power of $q$ or a power of $x$, and so has GK-dimension less than 5, respectively by Theorem~\ref{genpriv}(iii) and \ref{plusthm}(iv) or by Lemma~\ref{key}.
This forces $P/\mathfrak{p}\mathcal{D} = \{0\}$ and  so  $P=\mathfrak{p}\mathcal{D}$,  as claimed.

Suppose finally that $\mathfrak{p}= \mathfrak{p}_1$ or $\mathfrak{p} = \mathfrak{p}_2$. In the first case, since $q$ is a normal element of $\mathcal{D}$ by Lemma~\ref{normal}, $q \in \sqrt{\mathfrak{p}\mathcal{D}}$. Thus 
\begin{equation}\label{contain} q\mathcal{D} \; \subseteq \; P.
\end{equation}
We claim that (\ref{contain}) is an equality. To see this, note that $s \notin P$, since otherwise $P \cap Z(\mathcal{D})  = \mathfrak{m}_0$, which is ruled out by hypothesis. Moreover $\{x^j : j \geq 0\} \cap P = \emptyset$ by Lemma~\ref{key}. So we can localise at the Ore set $B = \{s^ix^j : i, j \geq 0\}$ of Definition~\ref{Ore} and pass to the localised Weyl algebra $D_{(B)} = A_2^{(B)}(k) \otimes S_{(B)}$ of Theorem~\ref{Weyls}. However, $P\mathcal{D}_{(B)}$ and $q\mathcal{D}_{(B)}$ have the same intersection with the centre $S_{(B)}$, namely $\omega\theta^{-1}S_{(B)} = \mathfrak{p}_1 S_{(B)}$. Therefore $P\mathcal{D}_{(B)} = q\mathcal{D}_{(B)}$ since the ideals of $D_{(B)}$ are centrally generated. Therefore $P/q\mathcal{D}$ is $B$-torsion, so, if it is not zero, it contains a nonzero element which is either killed by $q$ and by $s$, or by $q$ and $x$. As in the previous paragraph $\mathcal{D}/q\mathcal{D}$ is GK-homogeneous of GK-dimension 5, and so has no such non-zero torsion submodule, proving that (\ref{contain}) is an equality.

If  $\mathfrak{p} = \mathfrak{p}_2$ then the argument to show that $P = s\mathcal{D}$ is similar, but using the Ore set $A$; it is left to the reader.
\end{proof}

\medskip

\subsection{The Dixmier-Moeglin equivalence}\label{DM}

The following gives evidence in favour of \cite[Conjecture~1.3]{BL}, which proposes that an affine noetherian Hopf $\mathbb{C}$-algebra of finite
GK dimension should satisfy the Dixmier-Moeglin equivalence. See \cites{Be, BG} for definitions and background.

\begin{corollary}\label{DMcor} $\mathcal{D}$ satisfies the Dixmier-Moeglin equivalence.
\end{corollary}
\begin{proof} We check first using the description of the primitive spectrum in $\S$\ref{privshape} that every primitive ideal is locally closed.  For classes {\bf (I)} and {\bf (III)} this is clear since all these primitive ideals are maximal. The primitive ideals in {\bf (II)} are homeomorphic to the primitive spectrum of $U(\mathfrak{sl}(2,k))$, and the latter algebra satisfies the equivalence by \cite{Moe}. Thus, by \cite[Lemma~II.7.15]{BG}, it only remains to show that every rational  prime ideal $P$  is primitive,   where $P$ is rational if the centre of the Goldie quotient algebra of $\mathcal{D}/P$ is $k$. The non-primitive prime ideals are listed in Theorem~\ref{prime} and it is easy to check case by case that none of them is rational. 
\end{proof}

Corollary~\ref{DMcor} proves Theorem~\ref{main3}(c). With one exception, parts (a) and (b) of that theorem are proved in the  results of the last two sections that describe the prime ideals of $\mathcal{D}$. The exception is  the claim that all the completely prime factors of $\mathcal{D}$ (with the possible exception of $\mathcal{D}/P_0$, as noted in Remark~\ref{birational-remark}) are birationally equivalent to Weyl algebras. 
For  the primitive ideals $P$ strictly containing $\mathcal{O}(G)^+\mathcal{D}$  
this follows from  \cite[Remarque~7.1]{Di}. For the other prime ideals, this is clear from the description of the prime ideals in the last two sections.

Based on little more than the known results and counterexamples for group algebras and enveloping algebras, the theorem \cite{BL} for the cocommutative case, the recent work of Sierra and Walton on the noetherian property for enveloping algebras \cite{SW}, together with the above result and other isolated examples, we are tempted to propose the following conjecture as a strengthening in the pointed setting of \cite[Conjecture~1.3]{BL}, as much in the hope of stimulating the discovery of counterexamples as in expectation of a positive result.

\begin{conjecture}\label{wild} Let $H$ be an affine noetherian pointed Hopf $\mathbb{C}$-algebra. Then the following are equivalent:
\begin{enumerate}
\item[(1)] $\GKdim(H)$ is finite. 
\item[(2)] $H$ satisfies the Dixmier-Moeglin Equivalence.
\item[(3)] The group $G(H)$ of grouplikes of $H$ is nilpotent-by-finite.
\end{enumerate}
\end{conjecture}

Thanks to a famous result of Roseblade \cite{R} for group algebras, the implication $(2)\Longrightarrow(3)$ fails when $k$ is a finite field.

\bibliographystyle{amsplain}

\begin{bibdiv}
 \begin{biblist}
 
 \bib{ADPP}{article}{
   author= {Andruskiewitsch, N.},
   author={Dumas, F.}
   author={Pena Pollastri, H. M.},
   title={On the double of the Jordan plane},
   journal={Ark. Mat.},
   volume={60},
   date={2022},
   pages={213-229}
}
 


\bib{AP}{article}{
    author= {Andruskiewitsch, N.},
   author={Pena Pollastri, H. M.},
   title={On the restricted Jordan plane in odd characteristic},
   journal={J. Algebra  Appln.},
   volume={20},
   date={2021},
   number={2140012}
}

\bib{AP2}{article}{
    author= {Andruskiewitsch, N.},
   author={Pena Pollastri, H. M.},
   title={On the finite-dimensional representations of the double of the Jordan plane},
   journal={arXiv2211.01581},
   volume={},
   date={2022},
   number={}
}

\bib{ASZ}{article}{
author={Artin, M.},
author={Small, L. W.}
author={Zhang, J. J.}
title={Generic flatness for strongly noetherian rings},
journal={J. Algebra},
volume={221}
date={1999},
pages={579-610}
}

 \bib{Be}{article}{
    author= {Bell, J.},
   title={On the importance of being primitive},
   journal={Rev. Colombiana Mat.},
   volume={53},
   date={2019},
  pages={87-112}
}

 \bib{BL}{article}{
    author= {Bell, J.}
    author={Leung, W. H.},
   title={The Dixmier-Moeglin equivalence for cocommutative
Hopf Algebras of finite Gelfand-Kirillov Dimension},
   journal={Alg. Rep. Theory },
   volume={17},
   date={2014},
  pages={1843-1852}
}


 \bib{Bj}{article}{
    author= {Bjork, J.-E.},
   title={The Auslander condition on noetherian rings},
   journal={Seminaire Malliavin, Lecture Notes in Math.},
   volume={1404},
   date={1989},
  pages={137-173}
}


 \bib{B}{article}{
    author= {Brown, K. A.},
   title={Unruffled extensions and flatness over central subalgebras},
   journal={J. Algebra},
   volume={284},
   date={2005},
  pages={771-800}
}


 \bib{BCJ}{article}{
    author= {Brown, K. A.},
    author={Couto, M.}
    author={Jahn, A.},
   title={The finite dual of commutative-by-finite Hopf algebras},
   journal={Glasgow Math. J.},
   volume={},
   date={2022},
  pages={1-28}
}


\bib{BG}{book}{
   author={Brown, K.  A.}
   author={Goodearl, K. R.},,
   title={Lectures on Algebraic Quantum Groups},
   series={Advanced Courses in Math. CRM Barcelona},
   volume={},
   edition={},
   publisher={Birkhauser},
   date={2002},
   pages={xxiv+348},
   isbn={3-7643-6714-8},
   review={},
   doi={},
}

\bib{BZ}{article}{
    author= {Brown, K. A.}
    author= {Zhang, J.  J.}
   title={Dualising complexes and twisted Hochschild
(co)homology for noetherian Hopf algebras},
   journal={J. Algebra},
   volume={320},
   date={2008},
  pages={1814-1850}
}


\bib{Cohn}{book}{
   author={Cohn, P.  M.},
   title={Algebra, Vol. II}, 
   publisher={Wiley},
   date={1977},
  pages={},
  isbn={},
  review={},
   doi={},
}

\bib{Di}{article}{
    author= {Dixmier, J.}
   title={Quotients simples de l'alg$\grave{\mathrm{e}}$bre enveloppante de $\mathfrak{sl}_2$},
   journal={J. Algebra},
   volume={24},
   date={1973},
  pages={551-564}
}

\bib{DT}{article}{
    author= {Doi, Y.}
    author= {Takeuchi, M.},
   title={Multiplication alteration by two-cocycles - the quantum version},
   journal={Comm. in Algebra},
   volume={22},
   date={1994},
  pages={5715-5732}
}
\bib{F}{book}{
   author={Fossum, R.  M.},
   title={The Divisor Class Group of a Krull Domain},
   series={Ergebnisse der Mathematik und ihrer Grenzgebiete},
   volume={74},
   edition={},
   publisher={Springer},
   date={1973},
   pages={148},
   isbn={3-540-06044-8},
   review={},
   doi={},
}

 \bib{GL}{article}{
    author= {Goodearl, K. R.}
    author={Lenagan, T. L.}
   title={Primitive ideals in quantum SL3 and GL3}, 
   journal={Contemp. Math.},
   volume={562},
   date={2012},
  pages={115-140}
}





\bib{GW}{book}{
   author={Goodearl, K. R.}
   author={Warfield, R. W.},
   title={An Introduction to Noncommutative Noetherian rings},
   series={London Math. Soc. Student Texts},
   volume={61},
   edition={Second edition},
   publisher={Cambridge University Press},
   date={2004},
   pages={xxiv+344},
   isbn={0-521-83687-5},
   doi={},
}
 

\bib{KL}{book}{
   author={Krause, G. R.},
   author={Lenagan, T. H.},
   title={Growth of Algebras and Gelfand-Kirillov Dimension},
   series={Graduate Studies in Math.},
   volume={22},
   edition={Revised Edition},
   publisher={Amer. Math. Soc.},
   date={2000},
   pages={212},
   isbn={0-8218-0859-1},
   review={},
   doi={},
}

 


\bib{Lev}{article}{
   author={Levasseur, T.},
   title={Some properties of non-commutative regular graded rings},
   journal={Glasgow Math. J.},
   volume={34},
   date={1992},
   pages={277-300}
}

 
\bib{LL}{article}{
    author= {Li, K.}
   author={Liu, G.},
   title={The finite duals of affine prime regular Hopf algebras of GK-dimension one},
   journal={arXiv2103.00495},
   volume={},
   date={2021}
   number={}
}

\bib{McCR}{book}{
   author={McConnell, J. C.}
   author={Robson, J. C.},
   title={Noncommutative Noetherian rings},
   series={Graduate Studies in Mathematics},
   volume={30},
   edition={Revised edition},
   publisher={Amer. Math. Soc., Providence, RI},
   date={2001},
   pages={xx+636},
   isbn={0-8218-2169-5},
   doi={10.1090/gsm/030},
}

\bib{McCSt}{article}{
   author={McConnell, J. C.}
   author={Stafford, J. T.},
   title={Gelfand-Kirillov dimension and associated graded modules},
   journal={J. Algebra},
   volume={125},
   date={1989},
   pages={197-214}
}

 

\bib{Moe}{article}{
   author={Moeglin, C.},
   title={Id\'eaux primitifs d\'es alg\`ebres enveloppantes},
   journal={J. Math. Pures Appl.},
   volume={59},
   date={1980},
   pages={265-336}
}

 


\bib{R}{article}{
   author={Roseblade, J. E.},
   title={Group rings of polycyclic groups},
   journal={J. Pure Appl. Algebra},
   volume={3},
   date={1973},
   pages={307-328}
}

\bib{SW}{article}{   author={Sierra, S.}   author={Walton, C.}
   title={The universal enveloping algebra of the Witt algebra is not noetherian},
   journal={Adv. Math.},   volume={262},   date={2014},   pages={239-260}}

 

\end{biblist}
\end{bibdiv}

\end{document}